\documentclass[12pt]{amsart}
\usepackage{amsmath,amsthm,amsfonts,amssymb,mathrsfs}
\date{\today}

\usepackage{color}

\newtheorem{theorem}{Theorem}[section]
\newtheorem{proposition}[theorem]{Proposition}
\newtheorem{corollary}[theorem]{Corollary}
\newtheorem{lemma}[theorem]{Lemma}

\theoremstyle{definition}
\newtheorem{example}[theorem]{Example}
\newtheorem{remark}[theorem]{Remark}
\newtheorem{definition}[theorem]{Definition}

\usepackage{hyperref}

\begin{document}

\title[Topological monoids of almost
monotone injective co-finite partial selfmaps]{Topological monoids
of almost monotone injective co-finite partial selfmaps of
positive integers}

\author[Ivan~Chuchman]{Ivan~Chuchman}
\address{Department of Mechanics and Mathematics, Ivan Franko Lviv National
University, Universytetska 1, Lviv, 79000, Ukraine}
\email{chuchman\underline{\hskip5pt}\,i@mail.ru}

\author[Oleg~Gutik]{Oleg~Gutik}
\address{Department of Mechanics and Mathematics, Ivan Franko Lviv National
University, Universytetska 1, Lviv, 79000, Ukraine}
\email{o\underline{\hskip5pt}\,gutik@franko.lviv.ua,
ovgutik@yahoo.com}



\begin{abstract}
In this paper we study the semigroup
$\mathscr{I}_{\infty}^{\,\Rsh\!\!\!\nearrow}(\mathbb{N})$ of partial
co-finite almost monotone bijective transformations of the set of
positive integers $\mathbb{N}$. We show that the semigroup
$\mathscr{I}_{\infty}^{\,\Rsh\!\!\!\nearrow}(\mathbb{N})$ has
algebraic properties similar to the bicyclic semigroup: it is
bisimple and all of its non-trivial group homomorphisms are either
isomorphisms or group homomorphisms. Also we prove that every Baire
topology $\tau$ on
$\mathscr{I}_{\infty}^{\,\Rsh\!\!\!\nearrow}(\mathbb{N})$ such that
$(\mathscr{I}_{\infty}^{\,\Rsh\!\!\!\nearrow}(\mathbb{N}),\tau)$ is
a semitopological semigroup is discrete, describe the closure of
$(\mathscr{I}_{\infty}^{\,\Rsh\!\!\!\nearrow}(\mathbb{N}),\tau)$ in
a topological semigroup and construct non-discrete Hausdorff
semigroup topologies on
$\mathscr{I}_{\infty}^{\,\Rsh\!\!\!\nearrow}(\mathbb{N})$.
\end{abstract}

\maketitle


\section{Introduction and preliminaries}

In this paper all spaces are assumed to be Hausdorff. Furthermore
we shall follow the terminology of \cite{CHK, CP, Engelking1989,
Ruppert1984}. By $\omega$ we shall denote the first infinite
cardinal.

An algebraic semigroup $S$ is called {\it inverse} if for any
element $x\in S$ there exists the unique $x^{-1}\in S$ such that
$xx^{-1}x=x$ and $x^{-1}xx^{-1}=x^{-1}$. The element $x^{-1}$ is
called the {\it inverse of} $x\in S$. If $S$ is an inverse
semigroup, then the function $\operatorname{inv}\colon S\to S$
which assigns to every element $x$ of $S$ its inverse element
$x^{-1}$ is called an {\it inversion}.

If $S$ is a semigroup, then by $E(S)$ we shall denote the
\emph{band} (i.~e. the subset of idempotents) of $S$. If the band
$E(S)$ is a non-empty subset of $S$, then the semigroup operation on
$S$ determines the partial order $\leqslant$ on $E(S)$: $e\leqslant
f$ if and only if $ef=fe=e$. This order is called {\em natural}. A
\emph{semilattice} is a commutative semigroup of idempotents. A
semilattice $E$ is called {\em linearly ordered} or \emph{chain} if
the semilattice operation admits a linear natural order on $E$. A
\emph{maximal chain} of a semilattice $E$ is a chain which is
properly contained in no other chain of $E$. The Axiom of Choice
implies the existence of maximal chains in any partially ordered
set. According to \cite[Definition~II.5.12]{Petrich1984} chain $L$
is called $\omega$-chain if $L$ is isomorphic to
$\{0,-1,-2,-3,\ldots\}$ with the usual order $\leqslant$. Let $E$ be
a semilattice and $e\in E$. We denote ${\downarrow} e=\{ f\in E\mid
f\leqslant e\}$ and ${\uparrow} e=\{ f\in E\mid e\leqslant f\}$. By
$(\mathscr{P}_{<\omega}(\mathbb{N}),\subseteq)$ we shall denote the
free semilattice with identity over the set of positive integers
$\mathbb{N}$.

If $S$ is a semigroup, then by $\mathscr{R}$, $\mathscr{L}$,
$\mathscr{D}$ and $\mathscr{H}$ the Green relations on $S$ (see
\cite{CP}):
\begin{align*}
    &\qquad a\mathscr{R}b \mbox{ if and only if } aS^1=bS^1;\\
    &\qquad a\mathscr{L}b \mbox{ if and only if } S^1a=S^1b;\\
    &\qquad \mathscr{D}=\mathscr{L}\circ\mathscr{R}=\mathscr{R}\circ\mathscr{L};\\
    &\qquad \mathscr{H}=\mathscr{L}\cap\mathscr{R}.
\end{align*}
A semigroup $S$ is called \emph{simple} if $S$ does not contain
proper two-sided ideals and \emph{bisimple} if all elements of $S$
are $\mathscr{D}$-equivalent.

A {\it semitopological} (resp. \emph{topological}) {\it semigroup}
is a topological space together with a separately (resp. jointly)
continuous semigroup operation.

Let $\mathscr{I}_\lambda$ denote the set of all partial one-to-one
transformations of a set $X$ of cardinality $\lambda$ together
with the following semigroup operation:
$x(\alpha\beta)=(x\alpha)\beta$ if
$x\in\operatorname{dom}(\alpha\beta)=\{
y\in\operatorname{dom}\alpha\mid
y\alpha\in\operatorname{dom}\beta\}$,  for
$\alpha,\beta\in\mathscr{I}_\lambda$. The semigroup
$\mathscr{I}_\lambda$ is called the \emph{symmetric inverse
semigroup} over the set $X$~(see \cite{CP}). The symmetric inverse
semigroup was introduced by Wagner~\cite{Wagner1952} and it plays
a major role in the theory of semigroups.

We denote
 $
 \mathscr{I}_\lambda^n=\{
\alpha\in\mathscr{I}_\lambda\mid \operatorname{rank}\alpha\leqslant
n\},
 $
{} for $n=1,2,3,\ldots$. Obviously, $\mathscr{I}_\lambda^n$
($n=1,2,3,\ldots$) is an inverse semigroup, $\mathscr{I}_\lambda^n$
is an ideal of $\mathscr{I}_\lambda$ for each $n=1,2,3,\ldots$.
Further, we shall call the semigroup $\mathscr{I}_\lambda^n$ the
\emph{symmetric inverse semigroup of finite transformations of the
rank $n$}.

Let $\mathbb{N}$ be the set of all positive integers. By
$\mathscr{I}_{\infty}^{\!\nearrow}(\mathbb{N})$ we shall denote
the semigroup of monotone, non-decreasing, injective partial
transformations of $\mathbb{N}$ such that the sets
$\mathbb{N}\setminus\operatorname{dom}\varphi$ and
$\mathbb{N}\setminus\operatorname{rank}\varphi$ are finite for all
$\varphi\in\mathscr{I}_{\infty}^{\!\nearrow}(\mathbb{N})$.
Obviously, $\mathscr{I}_{\infty}^{\!\nearrow}(\mathbb{N})$ is an
inverse subsemigroup of the semigroup $\mathscr{I}_\omega$. The
semigroup $\mathscr{I}_{\infty}^{\!\nearrow}(\mathbb{N})$ is
called \emph{the semigroup of co-finite monotone partial
bijections} of $\mathbb{N}$~\cite{GutikRepovs2010?}.

We shall denote every element $\alpha$ of the semigroup
$\mathscr{I}_{\omega}$ by
 $
\left(%
\begin{array}{ccccc}
  n_1 & n_2 & n_3 & n_4 & \ldots \\
  m_1 & m_2 & m_3 & m_4 & \ldots \\
\end{array}%
\right)
 $
 and this means that $\alpha$ maps the
positive integer $n_i$ into $m_i$ for all $i=1,2,3,\ldots.$ We
observe that an element $\alpha$ of the semigroup
$\mathscr{I}_{\omega}$ is an element of the semigroup
$\mathscr{I}_{\infty}^{\!\nearrow}(\mathbb{N})$ if and only if it
satisfies the following conditions:
\begin{itemize}
    \item[$(i)$] the sets
    $\mathbb{N}\setminus\{n_1,n_2,n_3,n_4,\ldots\}$
    and $\mathbb{N}\setminus\{m_1, m_2, m_3, m_4, \ldots\}$ are
    finite;  and
    \item[$(ii)$] $n_1<n_2<n_3<n_4<\ldots$ and
    $m_1<m_2<m_3<m_4<\ldots$~.
\end{itemize}

A partial map $\alpha\colon \mathbb{N}\rightharpoonup \mathbb{N}$
is called \emph{almost monotone} if there exists a finite subset
$A$ of $\mathbb{N}$ such that the restriction
$\alpha\mid_{\mathbb{N}\setminus A}\colon \mathbb{N}\setminus
A\rightharpoonup \mathbb{N}$ is a monotone partial map.

By $\mathscr{I}_{\infty}^{\,\Rsh\!\!\!\nearrow}(\mathbb{N})$ we
shall denote the semigroup of monotone, almost non-decreasing,
injective partial transformations of $\mathbb{N}$ such that the
sets $\mathbb{N}\setminus\operatorname{dom}\varphi$ and
$\mathbb{N}\setminus\operatorname{rank}\varphi$ are finite for all
$\varphi\in\mathscr{I}_{\infty}^{\,\Rsh\!\!\!\nearrow}(\mathbb{N})$.
Obviously,
$\mathscr{I}_{\infty}^{\,\Rsh\!\!\!\nearrow}(\mathbb{N})$ is an
inverse subsemigroup of the semigroup $\mathscr{I}_\omega$ and the
semigroup $\mathscr{I}_{\infty}^{\!\nearrow}(\mathbb{N})$ is an
inverse subsemigroup of
$\mathscr{I}_{\infty}^{\,\Rsh\!\!\!\nearrow}(\mathbb{N})$ too. The
semigroup
$\mathscr{I}_{\infty}^{\,\Rsh\!\!\!\nearrow}(\mathbb{N})$ is
called \emph{the semigroup of co-finite almost monotone partial
bijections} of $\mathbb{N}$. We observe that an element $\alpha$
of the semigroup $\mathscr{I}_{\omega}$ is an element of the
semigroup
$\mathscr{I}_{\infty}^{\,\Rsh\!\!\!\nearrow}(\mathbb{N})$ if and
only if it satisfies conditions $(i)$ and $(iii)$:
\begin{itemize}
    \item[$(iii)$] there exists a positive integer $i$ such that
    $n_i<n_{i+1}<n_{i+2}<n_{i+3}<\ldots$ and
    $m_i<m_{i+1}<m_{i+2}<m_{i+3}<\ldots$~.
\end{itemize}

Further by $\mathbb{I}$ we shall denote the identity of the
semigroup
$\mathscr{I}_{\infty}^{\,\Rsh\!\!\!\nearrow}(\mathbb{N})$.

The bicyclic semigroup ${\mathscr{C}}(p,q)$ is the semigroup with
the identity $1$ generated by elements $p$ and $q$ subject only to
the condition $pq=1$. The bicyclic semigroup is bisimple and every
one of its congruences is either trivial or a group congruence.
Moreover, every non-annihilating homomorphism $h$ of the bicyclic
semigroup is either an isomorphism or the image of
${\mathscr{C}}(p,q)$ under $h$ is a cyclic group~(see
\cite[Corollary~1.32]{CP}). The bicyclic semigroup plays an
important role in algebraic theory of semigroups and in the theory
of topological semigroups. For example the well-known result of
Andersen~\cite{Andersen} states that a ($0$--)simple semigroup is
completely ($0$--)simple if and only if it does not contain the
bicyclic semigroup. The bicyclic semigroup admits only the
discrete topology and a topological semigroup $S$ can contain
${\mathscr C}(p,q)$ only as an open
subset~\cite{EberhartSelden1969}. Neither stable nor
$\Gamma$-compact topological semigroups can contain a copy of the
bicyclic semigroup~\cite{AHK, HildebrantKoch1988}. Also, the
bicyclic semigroup does not embed into a countably compact
topological inverse semigroup~\cite{GutikRepovs2007}. Moreover, in
\cite{BanakhDimitrovaGutik2009} and
\cite{BanakhDimitrovaGutik20??} the conditions were given when a
countable compact or pseudocompact topological semigroup does not
contain the bicyclic semigroup. However, Banakh, Dimitrova and
Gutik constructed with set-theoretic assumptions (Continuum
Hypothesis or Martin Axiom) an example of a Tychonoff countable
compact topological semigroup which contains the bicyclic
semigroup~\cite{BanakhDimitrovaGutik20??}.

Many semigroup theorists have considered a topological semigroup
of (continuous) transformations of an arbitrary topological space.
Be\u{\i}da~\cite{Beida1980}, Orlov~\cite{Orlov1974, Orlov1974a},
and Subbiah~\cite{Subbiah1987} have considered semigroup and
inverse semigroup  topologies of semigroups of partial
homeomorphisms of some classes of topological spaces.

Gutik and Pavlyk  \cite{GutikPavlyk2005} considered the special
case of the semigroup $\mathscr{I}_\lambda^n$: an infinite
topological semigroup of $\lambda\times\lambda$-matrix units
$B_\lambda$. They showed that an infinite topological semigroup of
$\lambda\times\lambda$-matrix units $B_\lambda$ does not embed
into a compact topological semigroup and that $B_\lambda$ is
algebraically $h$-closed in the class of topological inverse
semigroups. They also described the Bohr compactification of
$B_\lambda$, minimal semigroup and minimal semigroup inverse
topologies on $B_\lambda$.

Gutik, Lawson and Repov\v{s} \cite{GutikLawsonRepovs2009}
introduced the notion of a semigroup with a tight ideal series and
investigated their closures in semitopological semigroups,
particularly inverse semigroups with continuous inversion. As a
corollary they showed that the symmetric inverse semigroup of
finite transformations $\mathscr{I}_\lambda^n$ of infinite
cardinal $\lambda$ is algebraically closed in the class of
(semi)topological inverse semigroups with continuous inversion.
They also derived related results about the nonexistence of
(partial) compactifications of classes of considered semigroups.

Gutik and Reiter \cite{GutikReiter2009} showed that the
topological inverse semigroup $\mathscr{I}_\lambda^n$ is
algebraically $h$-closed in the class of topological inverse
semigroups. They also proved that a topological semigroup $S$ with
countably compact square $S\times S$ does not contain the
semigroup $\mathscr{I}_\lambda^n$ for infinite cardinals $\lambda$
and showed that the Bohr compactification of an infinite
topological semigroup $\mathscr{I}_\lambda^n$ is the trivial
semigroup.

In \cite{GutikReiter2010?} Gutik and Reiter showed that the
symmetric inverse semigroup of finite transformations
$\mathscr{I}_\lambda^n$ of infinite cardinal $\lambda$ is
algebraically closed in the class of semitopological inverse
semigroups with continuous inversion. There they described all
congruences on the semigroup $\mathscr{I}_\lambda^n$ and all compact
and countably compact topologies $\tau$ on $\mathscr{I}_\lambda^n$
such that $(\mathscr{I}_\lambda^n,\tau)$ is a semitopological
semigroup.

Gutik, Pavlyk and Reiter \cite{GutikPavlykReiter2009} showed that a
topological semigroup of finite partial bijections
$\mathscr{I}_\lambda^n$ of infinite set with a compact subsemigroup
of idempotents is absolutely $H$-closed. They proved that no
Hausdorff countably compact topological semigroup and no Tychonoff
topological semigroup with pseudocompact square contain
$\mathscr{I}_\lambda^n$ as a subsemigroup. They proved that every
continuous homomorphism from topological semigroup
$\mathscr{I}_\lambda^n$ into a Hausdorff countably compact
topological semigroup or Tychonoff topological semigroup with
pseudocompact square is annihilating. Also they gave sufficient
conditions for a topological semigroup $\mathscr{I}_\lambda^1$ to be
non-$H$-closed and showed that the topological inverse semigroup
$\mathscr{I}_\lambda^1$ is absolutely $H$-closed if and only if the
band $E(\mathscr{I}_\lambda^1)$ is compact
\cite{GutikPavlykReiter2009}.

In \cite{GutikRepovs2010?} Gutik and Repov\v{s} studied the
semigroup $\mathscr{I}_{\infty}^{\!\nearrow}(\mathbb{N})$ of partial
cofinite monotone bijective transformations of the set of positive
integers $\mathbb{N}$. They showed that the semigroup
$\mathscr{I}_{\infty}^{\!\nearrow}(\mathbb{N})$ has algebraic
properties similar to the bicyclic semigroup: it is bisimple and all
of its non-trivial group homomorphisms are either isomorphisms or
group homomorphisms. They proved that every locally compact topology
$\tau$ on $\mathscr{I}_{\infty}^{\!\nearrow}(\mathbb{N})$ such that
$(\mathscr{I}_{\infty}^{\!\nearrow}(\mathbb{N}),\tau)$ is a
topological inverse semigroup, is discrete and describe the closure
of $(\mathscr{I}_{\infty}^{\!\nearrow}(\mathbb{N}),\tau)$ in a
topological semigroup.

We remark that the bicyclic semigroup is isomorphic to the
semigroup $\mathscr{C}_{\mathbb{N}}(\pi,\sigma)$ which is
generated by partial transformations $\pi$ and $\sigma$ of the set
of positive integers $\mathbb{N}$, defined as follows:
\begin{equation*}
    (n)\pi=n+1  \quad \mbox{ if } \, n\geqslant 1, \qquad
    \mbox{and} \qquad
    (n)\sigma=n-1   \quad \mbox{ if } \, n> 1.
\end{equation*}
Therefore the semigroup
$\mathscr{I}_{\infty}^{\!\nearrow}(\mathbb{N})$ contains an
isomorphic copy of the bicyclic semigroup ${\mathscr{C}}(p,q)$.

In the present paper we study the semigroup
$\mathscr{I}_{\infty}^{\,\Rsh\!\!\!\nearrow}(\mathbb{N})$ of partial
co-finite almost monotone bijective transformations of the set of
positive integers $\mathbb{N}$. We show that the semigroup
$\mathscr{I}_{\infty}^{\,\Rsh\!\!\!\nearrow}(\mathbb{N})$ has
algebraic properties similar to the bicyclic semigroup: it is
bisimple and all of its non-trivial group homomorphisms are either
isomorphisms or group homomorphisms. Also we prove that every Baire
topology $\tau$ on
$\mathscr{I}_{\infty}^{\,\Rsh\!\!\!\nearrow}(\mathbb{N})$ such that
$(\mathscr{I}_{\infty}^{\,\Rsh\!\!\!\nearrow}(\mathbb{N}),\tau)$ is
a semitopological semigroup is discrete, describe the closure of
$(\mathscr{I}_{\infty}^{\,\Rsh\!\!\!\nearrow}(\mathbb{N}),\tau)$ in
a topological semigroup and construct non-discrete Hausdorff
semigroup topologies on
$\mathscr{I}_{\infty}^{\,\Rsh\!\!\!\nearrow}(\mathbb{N})$.


\section{Algebraic properties of the semigroup
$\mathscr{I}_{\infty}^{\,\Rsh\!\!\!\nearrow}(\mathbb{N})$}

\begin{proposition}\label{proposition-2.1}
\begin{itemize}
    \item[$(i)$] An element $\alpha$ of the semigroup
         $\mathscr{I}_{\infty}^{\,\Rsh\!\!\!\nearrow}(\mathbb{N})$
         is an idempotent if and only if $(x)\alpha=x$ for every
         $x\in\operatorname{dom}\alpha$, and hence
         $E(\mathscr{I}_{\infty}^{\,\Rsh\!\!\!\nearrow}(\mathbb{N}))=
         E(\mathscr{I}_{\infty}^{\!\nearrow}(\mathbb{N}))$.

    \item[$(ii)$] If $\varepsilon,\iota\in
          E(\mathscr{I}_{\infty}^{\,\Rsh\!\!\!\nearrow}(\mathbb{N}))$,
          then $\varepsilon\leqslant\iota$ if and only if
          $\operatorname{dom}\varepsilon\subseteq
          \operatorname{dom}\iota$.

    \item[$(iii)$] The semilattice
          $E(\mathscr{I}_{\infty}^{\,\Rsh\!\!\!\nearrow}(\mathbb{N}))$
          is isomorphic to
          $(\mathscr{P}_{<\omega}(\mathbb{N}),\subseteq)$ under
          the mapping $(\varepsilon)h=\mathbb{N}\setminus
          \operatorname{dom}\varepsilon$.

    \item[$(iv)$] Every maximal chain in
          $E(\mathscr{I}_{\infty}^{\,\Rsh\!\!\!\nearrow}(\mathbb{N}))$
          is an $\omega$-chain.

    \item[$(v)$] For every $\varepsilon,\iota\in
          E(\mathscr{I}_{\infty}^{\,\Rsh\!\!\!\nearrow}(\mathbb{N}))$
          there exists $\alpha\in
          \mathscr{I}_{\infty}^{\,\Rsh\!\!\!\nearrow}(\mathbb{N})$
          such that $\alpha\alpha^{-1}=\varepsilon$ and
          $\alpha^{-1}\alpha=\iota$.

    \item[$(vi)$]
         $\mathscr{I}_{\infty}^{\,\Rsh\!\!\!\nearrow}(\mathbb{N})$
         is a simple semigroup.

    \item[$(vii)$] $\alpha\mathscr{R}\beta$ in
         $\mathscr{I}_{\infty}^{\,\Rsh\!\!\!\nearrow}(\mathbb{N})$
         if and only if
         $\operatorname{dom}\alpha=\operatorname{dom}\beta$.

    \item[$(viii)$] $\alpha\mathscr{L}\beta$ in
         $\mathscr{I}_{\infty}^{\,\Rsh\!\!\!\nearrow}(\mathbb{N})$
         if and only if
         $\operatorname{rank}\alpha=\operatorname{rank}\beta$.

    \item[$(ix)$] $\alpha\mathscr{H}\beta$ in
         $\mathscr{I}_{\infty}^{\,\Rsh\!\!\!\nearrow}(\mathbb{N})$
         if and only if
         $\operatorname{dom}\alpha=\operatorname{dom}\beta$ and
         $\operatorname{rank}\alpha=\operatorname{rank}\beta$.

    \item[$(x)$]
         $\mathscr{I}_{\infty}^{\,\Rsh\!\!\!\nearrow}(\mathbb{N})$
         is a bisimple semigroup.

\end{itemize}
\end{proposition}

\begin{proof}
Statements $(i)-(iv)$ are trivial and their proofs follow from the
definition of the semigroup
$\mathscr{I}_{\infty}^{\,\Rsh\!\!\!\nearrow}(\mathbb{N})$.

$(v)$ For the idempotents $\varepsilon=
\left(%
\begin{array}{ccccc}
  m_1 & m_2 & m_3 & m_4 & \ldots \\
  m_1 & m_2 & m_3 & m_4 & \ldots \\
\end{array}%
\right)
 $
and $\iota=
\left(%
\begin{array}{ccccc}
  l_1 & l_2 & l_3 & l_4 & \ldots \\
  l_1 & l_2 & l_3 & l_4 & \ldots \\
\end{array}%
\right)
 $
we put $\alpha=
\left(%
\begin{array}{ccccc}
  m_1 & m_2 & m_3 & m_4 & \ldots \\
  l_1 & l_2 & l_3 & l_4 & \ldots \\
\end{array}%
\right)
 $. Then $\alpha\alpha^{-1}=\varepsilon$ and
$\alpha^{-1}\alpha=\iota$.

$(vi)$ Let
 $
\alpha=
\left(%
\begin{array}{ccccc}
  n_1 & n_2 & n_3 & n_4 & \ldots \\
  m_1 & m_2 & m_3 & m_4 & \ldots \\
\end{array}%
\right)
 $
and
 $
\beta=
\left(%
\begin{array}{ccccc}
  k_1 & k_2 & k_3 & k_4 & \ldots \\
  l_1 & l_2 & l_3 & l_4 & \ldots \\
\end{array}%
\right)
 $
be any elements of the semigroup
$\mathscr{I}_{\infty}^{\,\Rsh\!\!\!\nearrow}(\mathbb{N})$, where
$n_i,m_i,k_i,l_i\in\mathbb{N}$ for $i=1,2,3,\ldots$~. We put
 $
\gamma=
\left(%
\begin{array}{ccccc}
  k_1 & k_2 & k_3 & k_4 & \ldots \\
  n_1 & n_2 & n_3 & n_4 & \ldots \\
\end{array}%
\right)
 $
and
 $
\delta=
\left(%
\begin{array}{ccccc}
  m_1 & m_2 & m_3 & m_4 & \ldots \\
  l_1 & l_2 & l_3 & l_4 & \ldots \\
\end{array}%
\right)
 $.
Then we have that $\gamma\alpha\delta=\beta$. Therefore
$\mathscr{I}_{\infty}^{\,\Rsh\!\!\!\nearrow}(\mathbb{N})
\cdot\alpha\cdot
\mathscr{I}_{\infty}^{\,\Rsh\!\!\!\nearrow}(\mathbb{N})=
\mathscr{I}_{\infty}^{\,\Rsh\!\!\!\nearrow}(\mathbb{N})$ for any
$\alpha\in\mathscr{I}_{\infty}^{\,\Rsh\!\!\!\nearrow}(\mathbb{N})$
and hence
$\mathscr{I}_{\infty}^{\,\Rsh\!\!\!\nearrow}(\mathbb{N})$ is a
simple semigroup.

$(vii)$ Let
$\alpha,\beta\in\mathscr{I}_{\infty}^{\,\Rsh\!\!\!\nearrow}(\mathbb{N})$
be such that $\alpha\mathscr{R}\beta$. Since
$\alpha\mathscr{I}_{\infty}^{\,\Rsh\!\!\!\nearrow}(\mathbb{N})=
\beta\mathscr{I}_{\infty}^{\,\Rsh\!\!\!\nearrow}(\mathbb{N})$ and
$\mathscr{I}_{\infty}^{\,\Rsh\!\!\!\nearrow}(\mathbb{N})$ is an
inverse semigroup, Theorem~1.17 \cite{CP} implies that
$\alpha\mathscr{I}_{\infty}^{\,\Rsh\!\!\!\nearrow}(\mathbb{N})=
\alpha\alpha^{-1}\mathscr{I}_{\infty}^{\,\Rsh\!\!\!\nearrow}(\mathbb{N})$,
$\beta\mathscr{I}_{\infty}^{\,\Rsh\!\!\!\nearrow}(\mathbb{N})=
\beta\beta^{-1}\mathscr{I}_{\infty}^{\,\Rsh\!\!\!\nearrow}(\mathbb{N})$
and $\alpha\alpha^{-1}=\beta\beta^{-1}$. Hence
$\operatorname{dom}\alpha=\operatorname{dom}\beta$.

Conversely, let
$\alpha,\beta\in\mathscr{I}_{\infty}^{\,\Rsh\!\!\!\nearrow}(\mathbb{N})$
be such that $\operatorname{dom}\alpha=\operatorname{dom}\beta$.
Then $\alpha\alpha^{-1}=\beta\beta^{-1}$. Since
$\mathscr{I}_{\infty}^{\,\Rsh\!\!\!\nearrow}(\mathbb{N})$ is an
inverse semigroup, Theorem~1.17 \cite{CP} implies that
$\alpha\mathscr{I}_{\infty}^{\,\Rsh\!\!\!\nearrow}(\mathbb{N})=
\alpha\alpha^{-1}\mathscr{I}_{\infty}^{\,\Rsh\!\!\!\nearrow}(\mathbb{N})=
\beta\mathscr{I}_{\infty}^{\,\Rsh\!\!\!\nearrow}(\mathbb{N})$ and
hence
$\alpha\mathscr{I}_{\infty}^{\,\Rsh\!\!\!\nearrow}(\mathbb{N})=
\beta\mathscr{I}_{\infty}^{\,\Rsh\!\!\!\nearrow}(\mathbb{N})$.

The proof of statement $(viii)$ is similar to $(vii)$.

Statement $(ix)$ follows from $(vii)$ and $(viii)$.

$(x)$ By statements $(vii)$ and $(viii)$ it is sufficient to show
that every distinct idempotents of the semigroup
$\mathscr{I}_{\infty}^{\,\Rsh\!\!\!\nearrow}(\mathbb{N})$ are
$\mathscr{D}$-equivalent. For the idempotents $\varepsilon=
\left(%
\begin{array}{ccccc}
  m_1 & m_2 & m_3 & m_4 & \ldots \\
  m_1 & m_2 & m_3 & m_4 & \ldots \\
\end{array}%
\right)
 $
and $\iota=
\left(%
\begin{array}{ccccc}
  l_1 & l_2 & l_3 & l_4 & \ldots \\
  l_1 & l_2 & l_3 & l_4 & \ldots \\
\end{array}%
\right)
 $
we put $\alpha=
\left(%
\begin{array}{ccccc}
  m_1 & m_2 & m_3 & m_4 & \ldots \\
  l_1 & l_2 & l_3 & l_4 & \ldots \\
\end{array}%
\right) $.
 Then by statements $(vii)$ and $(viii)$ we have that
$\varepsilon\mathscr{R}\alpha$ and $\alpha\mathscr{L}\iota$, and
hence $\varepsilon\mathscr{D}\iota$
\end{proof}

\begin{proposition}\label{proposition-2.2}
For every
$\alpha,\beta\in\mathscr{I}_{\infty}^{\,\Rsh\!\!\!\nearrow}(\mathbb{N})$,
both sets
 $
\{\chi\in\mathscr{I}_{\infty}^{\,\Rsh\!\!\!\nearrow}(\mathbb{N})\mid
\alpha\cdot\chi=\beta\}
 $
 and
 $
\{\chi\in\mathscr{I}_{\infty}^{\,\Rsh\!\!\!\nearrow}(\mathbb{N})\mid
\chi\cdot\alpha=\beta\}
 $
are finite. Consequently, every right translation and every left
translation by an element of the semigroup
$\mathscr{I}_{\infty}^{\,\Rsh\!\!\!\nearrow}(\mathbb{N})$ is a
finite-to-one map.
\end{proposition}

\begin{proof}
We denote
$A=\{\chi\in\mathscr{I}_{\infty}^{\,\Rsh\!\!\!\nearrow}(\mathbb{N})
\mid \alpha\cdot\chi=\beta\}$ and
$B=\{\chi\in\mathscr{I}_{\infty}^{\,\Rsh\!\!\!\nearrow}(\mathbb{N})
\mid \alpha^{-1}\cdot\alpha\cdot\chi=\alpha^{-1}\cdot\beta\}$.
Then $A\subseteq B$ and the restriction of any partial map
$\chi\in B$ to $\operatorname{dom}(\alpha^{-1}\cdot\alpha)$
coincides with the partial map $\alpha^{-1}\cdot\beta$. Since
every partial map from the semigroup
$\mathscr{I}_{\infty}^{\,\Rsh\!\!\!\nearrow}(\mathbb{N})$ is
almost monotone (i.~e., almost non-decreasing) and co-finite, the
set $B$ is finite and hence so is $A$.
\end{proof}

For an arbitrary non-empty set $X$ we denote by
$\emph{\textsf{S}}_\infty(X)$ the group of all bijective
transformations of $X$ with finite supports (i.~e.,
$\alpha\in\emph{\textsf{S}}_\infty(X)$ if and only if the set
$\{x\in X\mid(x)\alpha\neq x\}$ is finite).

The definition of the semigroup
$\mathscr{I}_{\infty}^{\,\Rsh\!\!\!\nearrow}(\mathbb{N})$ implies
the following proposition:

\begin{proposition}\label{proposition-2.3}
Every maximal subgroup of the semigroup
$\mathscr{I}_{\infty}^{\,\Rsh\!\!\!\nearrow}(\mathbb{N})$ is
isomorphic to $\textsf{S}_\infty(\mathbb{N})$.
\end{proposition}

The semigroup
$\mathscr{I}_{\infty}^{\,\Rsh\!\!\!\nearrow}(\mathbb{N})$ contains
$\mathscr{I}_{\infty}^{\!\nearrow}(\mathbb{N})$ as a subsemigroup
and Theorem~2.9 of~\cite{GutikRepovs2010?} states that if $S$ is a
semigroup and
$h\colon\mathscr{I}_{\infty}^{\!\nearrow}(\mathbb{N})\rightarrow S$
is a non-annihilating homomorphism, then either $h$ is a
monomorphism or $(\mathscr{I}_{\infty}^{\!\nearrow}(\mathbb{N}))h$
is a cyclic subgroup of $S$. This arises the following problem:
\emph{To describe all homomorphisms of the semigroup
$\mathscr{I}_{\infty}^{\,\Rsh\!\!\!\nearrow}(\mathbb{N})$.}

The definition of the semigroup
$\mathscr{I}_{\infty}^{\,\Rsh\!\!\!\nearrow}(\mathbb{N})$ implies
the following proposition:

\begin{proposition}\label{proposition-2.4}
For every $\gamma\in
\mathscr{I}_{\infty}^{\,\Rsh\!\!\!\nearrow}(\mathbb{N})$ there
exists $n_\gamma\in\mathbb{N}$ such that
$i-n_\gamma=(i)\alpha-(n_\gamma)\alpha$ for all $i\geqslant
n_\gamma$, $i\in\mathbb{N}$.
\end{proposition}

\begin{lemma}\label{lemma-2.5}
For every $\gamma\in
\mathscr{I}_{\infty}^{\,\Rsh\!\!\!\nearrow}(\mathbb{N})$ there
exists an idempotent $\varepsilon\in
\mathscr{C}_{\mathbb{N}}(\pi,\sigma)$ such that
$\gamma\cdot\varepsilon,\varepsilon\cdot\gamma\in
\mathscr{C}_{\mathbb{N}}(\pi,\sigma)$. Consequently, for every
idempotent $\iota\in
\mathscr{I}_{\infty}^{\,\Rsh\!\!\!\nearrow}(\mathbb{N})$ there
exists $\varepsilon_0\in E(\mathscr{C}_{\mathbb{N}}(\pi,\sigma))$
such that
$\iota\cdot\varepsilon_0=\varepsilon_0\cdot\iota=\varepsilon_0$.
\end{lemma}

\begin{proof}
Let $n_\gamma\in\mathbb{N}$ be such as in the statement of
Proposition~\ref{proposition-2.4}. We put
$m_\gamma=\max\{n_\gamma, (n_\gamma)\gamma\}$ and define
\begin{equation*}
    \varepsilon=
\left(%
\begin{array}{cccc}
  m_\gamma & m_\gamma+1 & m_\gamma+2 & \cdots \\
  m_\gamma & m_\gamma+1 & m_\gamma+2 & \cdots \\
\end{array}%
\right).
\end{equation*}
Then we have that
$\gamma\cdot\varepsilon,\varepsilon\cdot\gamma\in
\mathscr{C}_{\mathbb{N}}(\pi,\sigma)$.

Let $\iota$ be an arbitrary idempotent of the semigroup
$\mathscr{I}_{\infty}^{\,\Rsh\!\!\!\nearrow}(\mathbb{N})$. By the
first assertion of the lemma there exists $\varepsilon\in
E(\mathscr{C}_{\mathbb{N}}(\pi,\sigma))$ such that
$\iota\cdot\varepsilon\in\mathscr{C}_{\mathbb{N}}(\pi,\sigma)$.
Since the semigroup
$\mathscr{I}_{\infty}^{\,\Rsh\!\!\!\nearrow}(\mathbb{N})$ is inverse
Theorem~1.17~\cite{CP} implies that
$\varepsilon_0=\iota\cdot\varepsilon=\varepsilon\cdot\iota$ is an
idempotent of $\mathscr{C}_{\mathbb{N}}(\pi,\sigma)$. Hence we have
that
$\iota\cdot\varepsilon_0=\varepsilon_0\cdot\iota=\varepsilon_0$.
\end{proof}

\begin{lemma}\label{lemma-2.6}
Let $S$ be  a semigroup and $h\colon
\mathscr{I}_{\infty}^{\,\Rsh\!\!\!\nearrow}(\mathbb{N})
\rightarrow S$ be a non-annihilating homomorphism such that the
set $\big(E(\mathscr{C}_{\mathbb{N}}(\pi,\sigma))\big)h$ is
singleton. Then
$\big(\mathscr{I}_{\infty}^{\,\Rsh\!\!\!\nearrow}(\mathbb{N})\big)h
=\big(\mathscr{C}_{\mathbb{N}}(\pi,\sigma)\big)h$.
\end{lemma}

\begin{proof}
Suppose that $\big(E(\mathscr{C}_{\mathbb{N}}(\pi,\sigma))\big)h
=\{e\}$. Since
$\mathscr{I}_{\infty}^{\,\Rsh\!\!\!\nearrow}(\mathbb{N})$ is an
inverse semigroup and
$E(\mathscr{I}_{\infty}^{\,\Rsh\!\!\!\nearrow}(\mathbb{N}))=
E(\mathscr{C}_{\mathbb{N}}(\pi,\sigma))$ we conclude that $e$ is a
unique idempotent in
$\big(\mathscr{I}_{\infty}^{\,\Rsh\!\!\!\nearrow}(\mathbb{N})\big)h$.
Fix an arbitrary element $\gamma$ of
$\mathscr{I}_{\infty}^{\,\Rsh\!\!\!\nearrow}(\mathbb{N})$. Let
$\varepsilon$ be such as in Lemma~\ref{lemma-2.5}. Then we have
 \begin{equation*}
 (\gamma)h=(\gamma\cdot\gamma^{-1}\cdot\gamma)h=
 (\gamma)h\cdot(\gamma^{-1}\cdot\gamma)h=
 (\gamma)h\cdot(\varepsilon)h=(\gamma\cdot\varepsilon)h\in
 \big(\mathscr{C}_{\mathbb{N}}(\pi,\sigma)\big)h,
\end{equation*}
the assertion of the lemma holds.
\end{proof}

We need the following theorem from \cite{GutikRepovs2010?}:

\begin{theorem}[{\cite[Theorem~2.9]{GutikRepovs2010?}}]\label{Gutik-Repovs-theorem}
Let $S$ be a semigroup and
$h\colon\mathscr{I}_{\infty}^{\!\nearrow}(\mathbb{N})\rightarrow S$
a non-annihilating homomorphism. Then either $h$ is a monomorphism
or $(\mathscr{I}_{\infty}^{\!\nearrow}(\mathbb{N}))h$ is a cyclic
subgroup of $S$.
\end{theorem}

\begin{lemma}\label{lemma-2.7}
Let $S$ be  a semigroup and $h\colon
\mathscr{I}_{\infty}^{\,\Rsh\!\!\!\nearrow}(\mathbb{N}) \rightarrow
S$ be a homomorphism such that the restriction
$h|_{\mathscr{C}_{\mathbb{N}}(\pi,\sigma)}\colon
\mathscr{C}_{\mathbb{N}}(\pi,\sigma)\rightarrow
\big(\mathscr{C}_{\mathbb{N}}(\pi,\sigma)\big)h\subseteq S$ is an
isomorphism. Then $h$ is an isomorphism.
\end{lemma}

\begin{proof}
Suppose to the contrary that the map $h\colon
\mathscr{I}_{\infty}^{\,\Rsh\!\!\!\nearrow}(\mathbb{N})\rightarrow
S$ is not an isomorphism. Then by Theorem~\ref{Gutik-Repovs-theorem}
we have that the restriction
$h|_{\mathscr{I}_{\infty}^{\!\nearrow}(\mathbb{N})}\colon
\mathscr{I}_{\infty}^{\!\nearrow}(\mathbb{N})\rightarrow
(\mathscr{I}_{\infty}^{\!\nearrow}(\mathbb{N}))h\subseteq S$ is an
isomorphism. Since
$\mathscr{I}_{\infty}^{\,\Rsh\!\!\!\nearrow}(\mathbb{N})$ is an
inverse semigroup we conclude that if $(\alpha)h=(\beta)h$ for some
$\alpha,\beta\in\mathscr{I}_{\infty}^{\,\Rsh\!\!\!\nearrow}(\mathbb{N})$
then $\alpha\mathscr{H}\beta$. Otherwise if $\alpha$ and $\beta$ are
not $\mathscr{H}$-equivalent and $(\alpha)h\neq(\beta)h$ then
$(\alpha^{-1})h\neq(\beta^{-1})h$ and therefore either
$(\alpha\alpha^{-1})h\neq(\beta\beta^{-1})h$ or
$(\alpha^{-1}\alpha)h\neq(\beta^{-1}\beta)h$, a contradiction to the
assumption that the restriction
$h|_{\mathscr{I}_{\infty}^{\!\nearrow}(\mathbb{N})}\colon
\mathscr{I}_{\infty}^{\!\nearrow}(\mathbb{N})\rightarrow
(\mathscr{I}_{\infty}^{\!\nearrow}(\mathbb{N}))h\subseteq S$ is an
isomorphism. Thus by Green Theorem (see \cite[Theorem~2.20]{CP})
without loss of generality we can assume that
$(\mathbb{I})h=(\alpha)h$ for some $\alpha\in H(\mathbb{I})$. Since
the group $\textsf{S}_\infty(\mathbb{N})$ has only one proper normal
subgroup and such subgroup is the group
$\textsf{A}_\infty(\mathbb{N})$ of even permutations of $\mathbb{N}$
(see \cite{KarrasSolitar1956} and \cite[pp.~313--314,
Example]{Guran1981}) we conclude that
$(\textsf{A}_\infty(\mathbb{N}))h=(\mathbb{I})h$. We denote
\begin{equation*}
    \beta=
\left(%
\begin{array}{cccccccc}
  1 & 2 & 3 & 4 & 5 & \cdots & n & \cdots \\
  2 & 3 & 1 & 4 & 5 & \cdots & n & \cdots \\
\end{array}%
\right)
 \qquad \mbox{ and } \qquad
    \varepsilon_{1,2}=
\left(%
\begin{array}{cccccc}
  3 & 4 & 5 & \cdots & n & \cdots \\
  1 & 4 & 5 & \cdots & n & \cdots \\
\end{array}%
\right).
\end{equation*}
Then $\beta\in\textsf{A}_\infty(\mathbb{N})$. Therefore we have
that
\begin{equation*}
    (\varepsilon_{1,2})h=(\varepsilon_{1,2}\cdot\mathbb{I})h=
    (\varepsilon_{1,2})h\cdot(\mathbb{I})h=
    (\varepsilon_{1,2})h\cdot(\beta)h=
    (\varepsilon_{1,2}\cdot\beta)h
\end{equation*}
and similarly
$(\varepsilon_{1,2})h=(\beta\cdot\varepsilon_{1,2})h$. Since
\begin{equation*}
    \beta\cdot\varepsilon_{1,2}=
\left(%
\begin{array}{cccccc}
  2 & 4 & 5 & \cdots & n & \cdots \\
  3 & 4 & 5 & \cdots & n & \cdots \\
\end{array}%
\right)
 \qquad \mbox{ and } \qquad
    \varepsilon_{1,2}\cdot\beta=
\left(%
\begin{array}{cccccc}
  3 & 4 & 5 & \cdots & n & \cdots \\
  1 & 4 & 5 & \cdots & n & \cdots \\
\end{array}%
\right)
\end{equation*}
we conclude that
$\beta\cdot\varepsilon_{1,2},\varepsilon_{1,2}\cdot\beta\in
\mathscr{I}_{\infty}^{\!\nearrow}(\mathbb{N})$. Hence by
Theorem~\ref{Gutik-Repovs-theorem} the set
$(\mathscr{I}_{\infty}^{\!\nearrow}(\mathbb{N}))h$ contains only one
idempotent and therefore the assertions of Lemma~\ref{lemma-2.6}
hold. This completes the proof of the lemma.
\end{proof}

Theorem~\ref{Gutik-Repovs-theorem} and Lemmas~\ref{lemma-2.6} and
\ref{lemma-2.7} imply the following theorem:

\begin{theorem}\label{theorem-2.8}
Let $S$ be a semigroup and
$h\colon\mathscr{I}_{\infty}^{\,\Rsh\!\!\!\nearrow}(\mathbb{N})
\rightarrow S$  a non-annihilating homomorphism. Then either $h$
is a monomorphism or
$(\mathscr{I}_{\infty}^{\,\Rsh\!\!\!\nearrow}(\mathbb{N}))h$ is a
cyclic subgroup of $S$.
\end{theorem}

\section{Topologizations of some classes of countable semigroups}

\begin{definition}\label{definitio-3.1}
We shall say that a semigroup $S$ has:
\begin{itemize}
    \item an $\textsf{S}$-\emph{property} if for every $a,b\in S$
     there exist $c,d\in S^1$ such that $c\cdot a\cdot d=b$;

    \item an $\textsf{F}$-\emph{property} if for every $a,b,c,d\in S^1$ the
     sets $\{x\in S\mid a\cdot x=b\}$ and $\{x\in S\mid x\cdot
     c=d\}$ are finite or empty;

    \item an $\textsf{FS}$-\emph{property} if $S$ has
     $\textsf{F}$- and $\textsf{S}$-properties.
\end{itemize}
\end{definition}

\begin{remark}\label{remark-3.2}
We observe that
\begin{itemize}
    \item[1)] every simple (resp., left simple, right simple)
     semigroup has $\textsf{S}$-\emph{property};

    \item[2)] every free (Abelian) semigroup has
     $\textsf{F}$-\emph{property};

    \item[3)]
     $\mathscr{I}_{\infty}^{\,\Rsh\!\!\!\nearrow}(\mathbb{N})$,
     $\mathscr{I}_{\infty}^{\!\nearrow}(\mathbb{N})$ and
     the bicyclic semigroup have $\textsf{FS}$-property.
\end{itemize}
\end{remark}

\begin{lemma}\label{lemma-3.3}
Let $S$ be a Hausdorff semitopological semigroup with
$\textsf{FS}$-property. If $S$ has an isolated point then $S$ is
the discrete topological space.
\end{lemma}

\begin{proof}
Let $t$ be an isolated point in $S$. Since the semigroup $S$ has the
$\textsf{FS}$-\emph{property} we conclude that for every $s\in S$
there exist $a,b\in S^1$ such that $a\cdot s\cdot b=t$ and the
equation $a\cdot x\cdot b=t$ has a finite set of solutions.
Therefore the continuity of translations in $(S,\tau)$ implies that
the element $s$ has a finite open neighbourhood, and hence
Hausdorffness of $(S,\tau)$ implies that $s$ is an isolated point of
$(S,\tau)$. This completes the proof of the lemma.
\end{proof}

A topological space $X$ is called \emph{Baire} if for each sequence
$A_1, A_2,\ldots, A_i,\ldots$ of nowhere dense subsets of $X$ the
union $\displaystyle\bigcup_{i=1}^\infty A_i$ is a co-dense subset
of $X$~\cite{Engelking1989}.

\begin{theorem}\label{theorem-3.4}
Let $S$ be a countable semigroup with $\textsf{FS}$-property. Then
every Baire topology $\tau$ on $S$ such that $(S,\tau)$ is a
Hausdorff semitopological semigroup is discrete.
\end{theorem}

\begin{proof}
We consider countable cover $\Gamma=\{s\mid s\in S\}$ of the Baire
space $(S,\tau)$. Then there exists an isolated point $t$ in $S$.
By Lemma~\ref{lemma-3.3} the topological space is discrete.
\end{proof}

A Tychonoff space $X$ is called \emph{\v{C}ech complete} if for
every compactification $cX$ of $X$ the remainder $cX\setminus c(X)$
is an $F_\sigma$-set in $cX$~\cite{Engelking1989}.

Since every \v{C}ech complete space (and hence every locally
compact space) is Baire Theorem~\ref{theorem-3.4} implies the
following:

\begin{corollary}\label{corollary-3.5}
Every Hausdorff \v{C}ech complete (locally compact) countable
semitopological semigroup with $\textsf{FS}$-property is discrete.
\end{corollary}

A topological space $X$ is called \emph{hereditary Baire} if every
closed subset of $X$ is a Baire space~\cite{Engelking1989}. Every
\v{C}ech complete (and hence locally compact) space is hereditary
Baire (see \cite[Theorem~3.9.6]{Engelking1989}). We shall say that a
Hausdorff semitopological semigroup $S$ is an $I$-\emph{Baire space}
if either $sS$ or $Ss$ is a Baire space for every $s\in S$.

\begin{remark}\label{remark-3.6}
We observe that every left ideal $Ss$ and every right ideal $sS$ of
a regular semigroup $S$ are generated by some idempotents of $S$.
Therefore every principal left or right ideal of a regular Hausdorff
semitopological semigroup $S$ is a closed subset of $S$. Hence every
regular Hausdorff hereditary Baire semitopological semigroup is the
$I$-Baire space.
\end{remark}

\begin{theorem}\label{theorem-3.7}
Let $S$ be a countable semilattice with $\textsf{F}$-property.
Then every $I$-Baire topology $\tau$ on $S$ such that $(S,\tau)$
is a Hausdorff semitopological semilattice is discrete.
\end{theorem}

\begin{proof}
Let $s$ be an arbitrary element of the semilattice $S$. We consider
a countable cover $\Gamma=\{e\mid e\in sS\}$ of $sS$. Since
$(S,\tau)$ is an $I$-Baire space we conclude that there exists an
isolated point $t$ in $sS$. Since $S$ is a semilattice we have that
$s\cdot t=t$. Then ${\uparrow}_{sS}t=\{x\in sS\mid x\cdot t=t\}$ is
a finite subset of $S$ which contains $s$ and by
Proposition~VI-1.13~\cite{GHKLMS} we get that ${\uparrow}_{sS}t$ is
an open subset of $sS$. Hence there exists an open neighbourhood
$U(s)$ of $s$ in $S$ such that $U(s)\cap sS=\{s\}$. The continuity
of translations in $S$ implies that there exists an open
neighbourhood $V(s)\subseteq U(s)$ such that $V(s)\subseteq\{x\in
S\mid x\cdot s=s\}$. Since the semilattice $S$ is Hausdorff and has
$\textsf{F}$-property we have that $s$ is an isolated point of $S$.
\end{proof}

Theorem~\ref{theorem-3.7} implies the following:

\begin{corollary}\label{definitio-3.8}
Every $I$-Baire topology $\tau$ on the countable free semilattice
$FSL_\omega$ such that $(FSL_\omega,\tau)$ is a Hausdorff
semitopological semilattice is discrete.
\end{corollary}


\section{On topologizations and closures of the semigroup
$\mathscr{I}_{\infty}^{\,\Rsh\!\!\!\nearrow}(\mathbb{N})$}

Theorem~\ref{theorem-3.4} implies the following two corollaries:

\begin{corollary}\label{corollary-4.1}
Every Baire topology $\tau$ on
$\mathscr{I}_{\infty}^{\,\Rsh\!\!\!\nearrow}(\mathbb{N})$ such
that
$(\mathscr{I}_{\infty}^{\,\Rsh\!\!\!\nearrow}(\mathbb{N}),\tau)$
is a Hausdorff semitopological semigroup is discrete.
\end{corollary}

\begin{corollary}\label{corollary-4.2}
Every Baire topology $\tau$ on
$\mathscr{I}_{\infty}^{\!\nearrow}(\mathbb{N})$ such that
$(\mathscr{I}_{\infty}^{\!\nearrow}(\mathbb{N}),\tau)$ is a
Hausdorff semitopological semigroup is discrete.
\end{corollary}

We observe that Corollary~\ref{corollary-4.2} generalizes
Theorem~3.3 from \cite{GutikRepovs2010?}.

The following example shows that there exists a non-discrete
topology $\tau_F$ on the semigroup
$\mathscr{I}_{\infty}^{\,\Rsh\!\!\!\nearrow}(\mathbb{N})$ such
that
$(\mathscr{I}_{\infty}^{\,\Rsh\!\!\!\nearrow}(\mathbb{N}),\tau_F)$
is a Tychonoff topological inverse semigroup.

\begin{example}\label{example-4.2a}
We define a topology $\tau_F$ on the semigroup
$\mathscr{I}_{\infty}^{\,\Rsh\!\!\!\nearrow}(\mathbb{N})$ as
follows. For every
$\alpha\in\mathscr{I}_{\infty}^{\,\Rsh\!\!\!\nearrow}(\mathbb{N})$
we define a family
\begin{equation*}
    \mathscr{B}_F(\alpha)=\{U_\alpha(F)\mid F \mbox{ is a finite
    subset of } \operatorname{dom}\alpha\},
\end{equation*}
where
\begin{equation*}
    U_\alpha(F)=\{\beta\in\mathscr{I}_{\infty}^{\,\Rsh\!\!\!\nearrow}(\mathbb{N})
    \mid \operatorname{dom}\alpha=\operatorname{dom}\beta,
    \operatorname{ran}\alpha=\operatorname{ran}\beta \mbox{ and }
    (x)\beta=(x)\alpha \mbox{ for all } x\in F\}.
\end{equation*}
Since conditions (BP1)--(BP3)~\cite{Engelking1989} hold for the
family $\{\mathscr{B}_F(\alpha)\}_{\alpha\in
\mathscr{I}_{\infty}^{\,\Rsh\!\!\!\nearrow}(\mathbb{N})}$ we
conclude that the family $\{\mathscr{B}_F(\alpha)\}_{\alpha\in
\mathscr{I}_{\infty}^{\,\Rsh\!\!\!\nearrow}(\mathbb{N})}$ is the
base of the topology $\tau_F$ on the semigroup
$\mathscr{I}_{\infty}^{\,\Rsh\!\!\!\nearrow}(\mathbb{N})$.
\end{example}

\begin{proposition}\label{proposition-4.2b}
$(\mathscr{I}_{\infty}^{\,\Rsh\!\!\!\nearrow}(\mathbb{N}),\tau_F)$
is a Tychonoff topological inverse semigroup.
\end{proposition}

\begin{proof}
Let $\alpha$ and $\beta$ be arbitrary elements of the semigroup
$\mathscr{I}_{\infty}^{\,\Rsh\!\!\!\nearrow}(\mathbb{N})$. We put
$\gamma=\alpha\beta$ and let $F=\{n_1,\ldots,n_i\}$ be a finite
subset of $\operatorname{dom}\gamma$. We denote
$m_1=(n_1)\alpha,\ldots,m_i=(n_i)\alpha$ and
$k_1=(n_1)\gamma,\ldots,k_i=(n_i)\gamma$. Then we get that
$(m_1)\beta=k_1,\ldots,(m_i)\beta=k_i$. Hence we have that
\begin{equation*}
    U_\alpha(\{n_1,\ldots,n_i\})\cdot
    U_\beta(\{m_1,\ldots,m_i\})\subseteq
    U_\gamma(\{n_1,\ldots,n_i\})
\end{equation*}
and
\begin{equation*}
    \big(U_\gamma(\{n_1,\ldots,n_i\})\big)^{-1}\subseteq
    U_{\gamma^{-1}}(\{k_1,\ldots,k_i\}).
\end{equation*}
Therefore the semigroup operation and the inversion are continuous
in
$(\mathscr{I}_{\infty}^{\,\Rsh\!\!\!\nearrow}(\mathbb{N}),\tau_F)$.

We observe that the group of units $H(\mathbb{I})$ of the
semigroup
$\mathscr{I}_{\infty}^{\,\Rsh\!\!\!\nearrow}(\mathbb{N})$ with the
induced topology $\tau_F(H(\mathbb{I}))$ from
$(\mathscr{I}_{\infty}^{\,\Rsh\!\!\!\nearrow}(\mathbb{N}),\tau_F)$
is a topological group (see \cite[pp.~313--314,
Example]{Guran1981} and \cite{KarrasSolitar1956}) and the
definition of the topology $\tau_F$ implies that every
$\mathscr{H}$-class of the semigroup
$\mathscr{I}_{\infty}^{\,\Rsh\!\!\!\nearrow}(\mathbb{N})$ is an
open-and-closed subset of the topological space
$(\mathscr{I}_{\infty}^{\,\Rsh\!\!\!\nearrow}(\mathbb{N}),\tau_F)$.
Therefore Theorem~2.20~\cite{CP} implies that the topological
space
$(\mathscr{I}_{\infty}^{\,\Rsh\!\!\!\nearrow}(\mathbb{N}),\tau_F)$
is homeomorphic to a countable topological sum of topological
copies of $\big(H(\mathbb{I}),\tau_F(H(\mathbb{I}))\big)$. Since
every $T_0$-topological group is a Tychonoff topological space
(see \cite[Theorem~3.10]{Pontryagin1966} or
\cite[Theorem~8.4]{HewittRoos1963}) we conclude that the
topological space
$(\mathscr{I}_{\infty}^{\,\Rsh\!\!\!\nearrow}(\mathbb{N}),\tau_F)$
is Tychonoff too. This completes the proof of the proposition.
\end{proof}

\begin{remark}
We observe that the topology $\tau_F$ on
$\mathscr{I}_{\infty}^{\,\Rsh\!\!\!\nearrow}(\mathbb{N})$ induces
discrete topologies on the subsemigroups
$\mathscr{I}_{\infty}^{\!\nearrow}(\mathbb{N})$ and
$E(\mathscr{I}_{\infty}^{\,\Rsh\!\!\!\nearrow}(\mathbb{N}))$.
\end{remark}

\begin{example}\label{example-4.2c}
We define a topology $\tau_{W\!F}$ on the semigroup
$\mathscr{I}_{\infty}^{\,\Rsh\!\!\!\nearrow}(\mathbb{N})$ as
follows. For every
$\alpha\in\mathscr{I}_{\infty}^{\,\Rsh\!\!\!\nearrow}(\mathbb{N})$
we define a family
\begin{equation*}
    \mathscr{B}_{W\!F}(\alpha)=\{U_\alpha(F)\mid F \mbox{ is a finite
    subset of } \operatorname{dom}\alpha\},
\end{equation*}
where
\begin{equation*}
    U_\alpha(F)=\{\beta\in\mathscr{I}_{\infty}^{\,\Rsh\!\!\!\nearrow}(\mathbb{N})
    \mid \operatorname{dom}\beta\subseteq\operatorname{dom}\alpha \mbox{ and }
    (x)\beta=(x)\alpha \mbox{ for all } x\in F\}.
\end{equation*}
Since conditions (BP1)--(BP3)~\cite{Engelking1989} hold for the
family $\{\mathscr{B}_{W\!F}(\alpha)\}_{\alpha\in
\mathscr{I}_{\infty}^{\,\Rsh\!\!\!\nearrow}(\mathbb{N})}$ we
conclude that the family
$\{\mathscr{B}_{W\!F}(\alpha)\}_{\alpha\in
\mathscr{I}_{\infty}^{\,\Rsh\!\!\!\nearrow}(\mathbb{N})}$ is the
base of the topology $\tau_{W\!F}$ on the semigroup
$\mathscr{I}_{\infty}^{\,\Rsh\!\!\!\nearrow}(\mathbb{N})$.
\end{example}

\begin{proposition}\label{proposition-4.2d}
$(\mathscr{I}_{\infty}^{\,\Rsh\!\!\!\nearrow}(\mathbb{N}),\tau_{W\!F})$
is a Hausdorff topological inverse semigroup.
\end{proposition}

\begin{proof}
Let $\alpha$ and $\beta$ be arbitrary elements of the semigroup
$\mathscr{I}_{\infty}^{\,\Rsh\!\!\!\nearrow}(\mathbb{N})$. We put
$\gamma=\alpha\beta$ and let $F=\{n_1,\ldots,n_i\}$ be a finite
subset of $\operatorname{dom}\gamma$. We denote
$m_1=(n_1)\alpha,\ldots,m_i=(n_i)\alpha$ and
$k_1=(n_1)\gamma,\ldots,k_i=(n_i)\gamma$. Then we get that
$(m_1)\beta=k_1,\ldots,(m_i)\beta=k_i$. Hence we have that
\begin{equation*}
    U_\alpha(\{n_1,\ldots,n_i\})\cdot
    U_\beta(\{m_1,\ldots,m_i\})\subseteq
    U_\gamma(\{n_1,\ldots,n_i\})
\end{equation*}
and
\begin{equation*}
    \big(U_\gamma(\{n_1,\ldots,n_i\})\big)^{-1}\subseteq
    U_{\gamma^{-1}}(\{k_1,\ldots,k_i\}).
\end{equation*}
Therefore the semigroup operation and the inversion are continuous
in
$(\mathscr{I}_{\infty}^{\,\Rsh\!\!\!\nearrow}(\mathbb{N}),\tau_{W\!F})$.

Later we shall show that the topology $\tau_{W\!F}$ is Hausdorff.
Let $\alpha$ and $\beta$ be arbitrary distinct points of the space
$(\mathscr{I}_{\infty}^{\,\Rsh\!\!\!\nearrow}(\mathbb{N}),\tau_{W\!F})$.
Then only one of the following conditions holds:
\begin{itemize}
    \item[$(i)$]
    $\operatorname{dom}\alpha=\operatorname{dom}\beta$;

    \item[$(ii)$]
    $\operatorname{dom}\alpha\neq\operatorname{dom}\beta$.
\end{itemize}

In case $\operatorname{dom}\alpha=\operatorname{dom}\beta$ we have
that there exists $x\in\operatorname{dom}\alpha$ such that
$(x)\alpha\neq(x)\beta$. The definition of the topology
$\tau_{W\!F}$ implies that $U_\alpha(\{x\})\cap
U_\beta(\{x\})=\varnothing$.

If $\operatorname{dom}\alpha\neq\operatorname{dom}\beta$, then
only one of the following conditions holds:
\begin{itemize}
    \item[$(a)$]
    $\operatorname{dom}\alpha\subsetneqq\operatorname{dom}\beta$;

    \item[$(b)$]
    $\operatorname{dom}\beta\subsetneqq\operatorname{dom}\alpha$;

    \item[$(c)$]
    $\operatorname{dom}\alpha\setminus\operatorname{dom}\beta\neq\varnothing$
    and
    $\operatorname{dom}\beta\setminus\operatorname{dom}\alpha\neq\varnothing$.
\end{itemize}

Suppose that case $(a)$ holds. Let
$x\in\operatorname{dom}\beta\setminus\operatorname{dom}\alpha$ and
$y\in\operatorname{dom}\alpha$. The definition of the topology
$\tau_{W\!F}$ implies that  $U_\alpha(\{y\})\cap
U_\beta(\{x\})=\varnothing$.

Case $(b)$ is similar to $(a)$.

Suppose that case $(c)$ holds. Let
$x\in\operatorname{dom}\beta\setminus\operatorname{dom}\alpha$ and
$y\in\operatorname{dom}\alpha\setminus\operatorname{dom}\beta$. The
definition of the topology $\tau_{W\!F}$ implies that
$U_\alpha(\{y\})\cap U_\beta(\{x\})=\varnothing$.

This completes the proof of the proposition.
\end{proof}

\begin{remark}
We observe that the topology $\tau_{W\!F}$ on
$\mathscr{I}_{\infty}^{\,\Rsh\!\!\!\nearrow}(\mathbb{N})$ induces
non-discrete topologies on the semigroup
$\mathscr{I}_{\infty}^{\!\nearrow}(\mathbb{N})$ and the semilattice
$E(\mathscr{I}_{\infty}^{\,\Rsh\!\!\!\nearrow}(\mathbb{N}))$.
Moreover, every $\mathscr{H}$-class of the semigroup
$(\mathscr{I}_{\infty}^{\,\Rsh\!\!\!\nearrow}(\mathbb{N}),\tau_{W\!F})$
is homeomorphic to every $\mathscr{H}$-class of the semigroup
$(\mathscr{I}_{\infty}^{\,\Rsh\!\!\!\nearrow}(\mathbb{N}),\tau_{W})$
\end{remark}

The proof of the following proposition is similar to
Theorem~\ref{theorem-3.4}:

\begin{proposition}\label{proposition-4.3}
Every Hausdorff Baire topology $\tau$ on a countable group $G$
such that left (right) translations in $(G,\tau)$ are continuous
is discrete.
\end{proposition}

\begin{theorem}\label{theorem-4.4}
Let $S$ be a topological semigroup which contains a dense discrete
subspace $A$ such that every equations $a\cdot x=b$ and $y\cdot
c=d$ have finitely many solutions in $A$. Then $I=S\setminus A$ is
an ideal of $S$.
\end{theorem}

\begin{proof}
Suppose that $I$ is not an ideal of $S$. Then at least one of the
following conditions holds:
\begin{equation*}
    1)~IA\nsubseteq I, \qquad 2)~AI\nsubseteq I, \qquad \mbox{or}
    \qquad 3)~II\nsubseteq I.
\end{equation*}
Since $A$ is a discrete dense subspace of $S$,
Theorem~3.5.8~\cite{Engelking1989} implies that $A$ is an open
subspace of $S$. Suppose there exist $a\in A$ and $b\in I$ such
that $b\cdot a=c\notin I$. Since $A$ is a dense open discrete
subspace of $S$ the continuity of the semigroup operation in $S$
implies that there exists an open neighbourhood $U(b)$ of $b$ in
$S$ such that $U(b)\cdot \{a\}=\{c\}$. But the equation $a\cdot
x=b$ and $y\cdot c=d$ have finitely many solutions in $A$. This
contradicts to the assumption that $a,b\in S\setminus A$.
Therefore $b\cdot a=c\in I$ and hence $IA\subseteq I$. The proof
of the inclusion $AI\subseteq I$ is similar.

Suppose there exist $a,b\in I$ such that $a\cdot b=c\notin I$. Since
$A$ is a dense open discrete subspace of $S$ the continuity of the
semigroup operation in $S$ implies that there exist open
neighbourhoods $U(a)$ and $U(b)$ of $a$ and $b$ in $S$,
respectively, such that $U(a)\cdot U(b)=\{c\}$. But the equations
$x\cdot a=c$ and $y\cdot c=d$ have finitely many solutions in $A$.
This contradicts to the assumption that $a,b\in S\setminus A$.
Therefore $a\cdot b=c\in I$ and hence $II\subseteq I$.
\end{proof}

Theorem~\ref{theorem-4.4} implies
Corollaries~\ref{corollary-4.5} and \ref{corollary-4.6}:

\begin{corollary}\label{corollary-4.5}
Let $S$ be a topological semigroup which contains a dense discrete
subsemigroup
$\mathscr{I}_{\infty}^{\,\Rsh\!\!\!\nearrow}(\mathbb{N})$. If
$I=S\setminus\mathscr{I}_{\infty}^{\,\Rsh\!\!\!\nearrow}(\mathbb{N})
\neq\varnothing$ then $I$ is an ideal of $S$.
\end{corollary}

\begin{corollary}[\cite{GutikRepovs2010?}]\label{corollary-4.6}
Let $S$ be a topological semigroup which contains a dense discrete
subsemigroup $\mathscr{I}_{\infty}^{\!\nearrow}(\mathbb{N})$. If
$I=S\setminus\mathscr{I}_{\infty}^{\!\nearrow}(\mathbb{N})
\neq\varnothing$ then $I$ is an ideal of $S$.
\end{corollary}

\begin{proposition}\label{proposition-4.7}
Let $S$ be a topological semigroup which contains a dense discrete
subsemigroup
$\mathscr{I}_{\infty}^{\,\Rsh\!\!\!\nearrow}(\mathbb{N})$. Then for
every $c\in\mathscr{I}_{\infty}^{\,\Rsh\!\!\!\nearrow}(\mathbb{N})$
the set
\begin{equation*}
    D_c(A)=\{(x,y)\in\mathscr{I}_{\infty}^{\,\Rsh\!\!\!\nearrow}(\mathbb{N})
    \times\mathscr{I}_{\infty}^{\,\Rsh\!\!\!\nearrow}(\mathbb{N})\mid x\cdot y=c\}
\end{equation*}
is a closed-and-open subset of $S\times S$.
\end{proposition}

\begin{proof}
Since $\mathscr{I}_{\infty}^{\,\Rsh\!\!\!\nearrow}(\mathbb{N})$ is a
discrete subspace of $S$ we have that $D_c(A)$ is an open subset of
$S\times S$.

Suppose that there exists
$c\in\mathscr{I}_{\infty}^{\,\Rsh\!\!\!\nearrow}(\mathbb{N})$ such
that $D_c(A)$ is a non-closed subset of $S\times S$. Then there
exists an accumulation point $(a,b)\in S\times S$ of the set
$D_c(A)$. The continuity of the semigroup operation in $S$ implies
that $a\cdot b=c$. But
$\mathscr{I}_{\infty}^{\,\Rsh\!\!\!\nearrow}(\mathbb{N})\times
\mathscr{I}_{\infty}^{\,\Rsh\!\!\!\nearrow}(\mathbb{N})$ is a
discrete subspace of $S\times S$ and hence by
Corollary~\ref{corollary-4.5} the points $a$ and $b$ belong to the
ideal $I=S\setminus
\mathscr{I}_{\infty}^{\,\Rsh\!\!\!\nearrow}(\mathbb{N})$ and hence
$p\cdot q\in S\setminus
\mathscr{I}_{\infty}^{\,\Rsh\!\!\!\nearrow}(\mathbb{N})$ cannot be
equal to $c$.
\end{proof}

A topological space $X$ is defined to be \emph{pseudocompact} if
each locally finite open cover of $X$ is finite. According to
\cite[Theorem~3.10.22]{Engelking1989} a Tychonoff topological space
$X$ is pseudocompact if and only if each continuous real-valued
function on $X$ is bounded.

\begin{theorem}\label{theorem-4.8}
If a topological semigroup $S$ contains
$\mathscr{I}_{\infty}^{\,\Rsh\!\!\!\nearrow}(\mathbb{N})$ as a dense
discrete subsemigroup then the square $S\times S$ is not
pseudocompact.
\end{theorem}

\begin{proof}
Since the square $S\times S$ contains an infinite closed-and-open
discrete subspace $D_c(A)$, we conclude that $S\times S$ fails to be
pseudocompact (see \cite[Ex.~3.10.F(d)]{Engelking1989} or
\cite{Colmez1951}).
\end{proof}

\begin{remark}\label{remark-4.9}
Recall that, a topological semigroup $S$ is called $\Gamma$-compact
if for every $x\in S$ the closure of the set $\{x,x^2,x^3,\ldots\}$
is a compactum in $S$ (see \cite{HildebrantKoch1988}). Since the
semigroup $\mathscr{I}_{\infty}^{\!\nearrow}(\mathbb{N})$ contains
the bicyclic semigroup as a subsemigroup the results obtained in
\cite{AHK}, \cite{BanakhDimitrovaGutik2009},
\cite{BanakhDimitrovaGutik20??}, \cite{GutikRepovs2007},
\cite{HildebrantKoch1988} imply that \emph{if a topological
semigroup $S$ satisfies one of the following conditions: $(i)$~$S$
is compact; $(ii)$~$S$ is $\Gamma$-compact; $(iii)$~the square
$S\times S$ is countably compact; $(iv)$~$S$ is a countably compact
topological inverse semigroup; or $(v)$~the square $S\times S$ is a
Tychonoff pseudocompact space, then $S$ does not contain the
semigroup $\mathscr{I}_{\infty}^{\!\nearrow}(\mathbb{N})$ and hence
the semigroup
$\mathscr{I}_{\infty}^{\,\Rsh\!\!\!\nearrow}(\mathbb{N})$.}
\end{remark}

The proof of the following theorem is similar to
Theorem~\ref{theorem-4.8}:

\begin{theorem}\label{theorem-4.10}
If a topological semigroup $S$ contains
$\mathscr{I}_{\infty}^{\!\nearrow}(\mathbb{N})$ as a dense discrete
subsemigroup then the square $S\times S$ is not pseudocompact.
\end{theorem}

\end{document}